\documentclass[11pt]{article}
\usepackage[margin=1in]{geometry} 
\usepackage{amssymb,amsfonts,amsmath,bbm,mathrsfs,stmaryrd}
\usepackage{url}

\usepackage[colorlinks,
             linkcolor=black!75!red,
             citecolor=blue,
             pdftitle={Centers, cocenters and simple quantum groups},
             pdfauthor={Alexandru Chirvasitu},
             pdfsubject={rings and algebras},
             pdfkeywords={qa.quantum algebra, rt.representation theory, co.combinatorics},
             pdfproducer={pdfLaTeX},
             pdfpagemode=None,
             bookmarksopen=true
             bookmarksnumbered=true]{hyperref}

\usepackage{tikz}
\usetikzlibrary{arrows,decorations.pathreplacing,decorations.markings,shapes.geometric,through,fit,shapes.symbols,positioning}

\usepackage[amsmath,thmmarks,hyperref]{ntheorem}
\usepackage{cleveref}

\crefname{section}{Section}{Sections}
\crefformat{section}{#2Section~#1#3} 
\Crefformat{section}{#2Section~#1#3} 

\crefname{subsection}{\S}{\S\S}
\crefformat{subsection}{#2\S#1#3} 
\Crefformat{subsection}{#2\S#1#3} 

\theoremstyle{change}

\newtheorem{lemma}{Lemma}[section]
\newtheorem{proposition}[lemma]{Proposition}
\newtheorem{corollary}[lemma]{Corollary}
\newtheorem{theorem}[lemma]{Theorem}

\theoremstyle{nonumberplain}

\theoremstyle{change}
\theorembodyfont{\upshape}
\theoremsymbol{\ensuremath{\blacklozenge}}

\newtheorem{de}[lemma]{Definition}
\newtheorem{ex}[lemma]{Example}
\newtheorem{re}[lemma]{Remark}

\crefname{definition}{definition}{definitions}
\crefformat{definition}{#2definition~#1#3} 
\Crefformat{definition}{#2Definition~#1#3} 

\crefname{lemma}{lemma}{lemmas}
\crefformat{lemma}{#2lemma~#1#3} 
\Crefformat{lemma}{#2Lemma~#1#3} 

\crefname{proposition}{proposition}{propositions}
\crefformat{proposition}{#2proposition~#1#3} 
\Crefformat{proposition}{#2Proposition~#1#3} 

\crefname{ex}{example}{examples}
\crefformat{example}{#2example~#1#3} 
\Crefformat{example}{#2Example~#1#3} 

\crefname{remark}{remark}{remarks}
\crefformat{remark}{#2remark~#1#3} 
\Crefformat{remark}{#2Remark~#1#3} 

\crefname{corollary}{corollary}{corollaries}
\crefformat{corollary}{#2corollary~#1#3} 
\Crefformat{corollary}{#2Corollary~#1#3} 

\crefname{theorem}{theorem}{theorems}
\crefformat{theorem}{#2theorem~#1#3} 
\Crefformat{theorem}{#2Theorem~#1#3} 

\crefname{equation}{}{}
\crefformat{equation}{(#2#1#3)} 
\Crefformat{equation}{(#2#1#3)}

\theoremstyle{nonumberplain}
\theoremsymbol{\ensuremath{\blacksquare}}

\newtheorem{proof}{Proof}
\newtheorem{proof of main}{Proof of \Cref{th.main}}
\newtheorem{proof of simplfree}{Proof of \Cref{th.simplicity_free_prod}}

\DeclareMathOperator{\id}{id}

\newcommand\bZ{{\mathbb Z}}

\newcommand\bR{{\mathbb R}}
\newcommand\bC{{\mathbb C}}

\newcommand\cC{{\mathcal C}}
\newcommand\cM{{\mathcal M}}

\DeclareMathOperator{\lker}{\cat{LKer}}
\DeclareMathOperator{\rker}{\cat{RKer}}
\DeclareMathOperator{\hker}{\cat{HKer}}
\newcommand{\define}[1]{{\em #1}}

\newcommand{\cat}[1]{\textsc{#1}}

\newcommand{\qedhere}{\mbox{}\hfill\ensuremath{\blacksquare}}

\title{Centers, cocenters and simple quantum groups}
\author{Alexandru Chirvasitu\footnote{University of California at Berkeley, \url{chirvasitua@math.berkeley.edu}}}
\begin{document}
\maketitle

\begin{abstract}
We define the notion of a (linearly reductive) center for a linearly reductive quantum group, and show that the quotient of a such a quantum group by its center is simple whenever its fusion semiring is free in the sense of Banica and Vergnioux. We also prove that the same is true of free products of quantum groups under very mild non-degeneracy conditions. Several natural families of compact quantum groups, some with non-commutative fusion semirings and hence very ``far from classical'', are thus seen to be simple. Examples include quotients of free unitary groups by their centers, recovering previous work, as well as quotients of quantum reflection groups by their centers. 
\end{abstract}


\tableofcontents

%%%%%%%%%%%%%%%%%%%%%%%%%%%%%%%%%%%%%%%%%%%%%%%%%%%%%%%%%%%%%%%%%%%%%%%%%%%%%%%%%%%%%%%%%%%%%%%%%%%%%%%%%%%%%%%%%%
%%%%%%%%%%%%%%%%%%%%%%%%%%%%%%%%%%%%%%%%%%%%%%%%%%%%%%%%%%%%%%%%%%%%%%%%%%%%%%%%%%%%%%%%%%%%%%%%%%%%%%%%%%%%%%%%%%

\section*{Introduction}

Quantum groups of function algebra type, in the spirit of, say, \cite{MR901157,MR1016381,MR1015339}, appear naturally as Hopf algebras coacting universally on various structures (such as a vector space endowed with an $R$-matrix in the case of \cite{MR1015339}, for example, or a quadratic algebra in \cite{MR1016381}). The point of view adopted in this paper is that of the sources just cited (and dual to that of e.g. \cite{MR797001,MR934283}): Hopf algebras are regarded as algebras of (appropriately nice) functions on fictitious objects referred to as quantum groups. In fact, the only Hopf algebras we consider here are cosemisimple, in the sense that their categories of comodules (to be regarded as representations of the underlying quantum group) are semisimple; the phrase `linearly reductive' in the abstract refers to this semisimplicity property, by analogy with the use of the term `linearly reductive' in the context of algebraic groups. For compilations of references available at the time and overviews of the subject the reader can consult, for example, \cite[$\S$7]{MR1358358} or \cite[$\S$9]{MR1492989}. 

The theory of compact quantum groups initiated by Woronowicz in \cite{MR901157} (in non-quite-final form in this early paper) has led to the discovery of numerous examples that do not arise as deformed function algebras of ordinary Lie groups, and to an explosion in the field. The papers \cite{MR2511633,MR2504527,Wang_new} provide a wealth of information, as do the references therein. For more general surveys of the (already vast, at the time) literature on compact quantum groups we refer to \cite{MR1616348,MR1741102}. We are interested here in the purely algebraic counterpart of Woronowicz's $C^*$-algebraic notion of a compact quantum group, i.e. in the so-called CQG algebras of \cite{MR1310296} (see \Cref{se.prel}), and in fact even more generally, in cosemisimple Hopf algebras. 

In the course of trying to push analogies to ordinary compact groups as far as possible, one might be led quite naturally to ask what a \define{simple} compact (or more generaly, linearly reductive) quantum group is. The notion was introduced in \cite{MR2504527}, where Wang also shows that some of the well-known examples in the literature are simple (e.g. deformed function algebras of simple compact Lie groups). One problem posed in \cite{MR2504527} is to provide examples of simple compact quantum groups with non-commutative fusion ring (i.e. Grothendieck ring of the category of comodules; see \Cref{se.prel} for terminology); this is addressed in \cite{chirvasitu:123509}, where it is shown that free unitary groups, which are to compact quantum groups what ordinary unitary groups are to compact Lie groups, become simple once we quotient out a one-dimensional central torus (see \cite[Theorem 1]{chirvasitu:123509}, and below). 

The present paper is in a sense a sequel to \cite{chirvasitu:123509}, and is concerned with extending the earlier results to a wider class of linearly reductive quantum groups. It is organized as follows: 

\Cref{se.prel} is devoted to the preparations needed afterwards. We fix some notations, recall some reoccuring specific examples of compact quantum groups, review useful notions such as that of free fusion semiring, etc.

In \Cref{se.center} we introduce the notion of center for a linearly reductive quantum group. The idea is simple enough, and is based on the representation-theoretic characterization of the center of a compact Lie group (e.g. \cite{MR2130607}). We also introduce the quotient of a quantum group by its center, referring to it as the cocenter (\Cref{def.center_cocenter}). The main result of the section is \Cref{pr.alt_descr}, which gives a convenient description of the cocenter and will come up again and again in later sections.  

The first main result of the paper is \Cref{th.main}, which says that cocenters of quantum groups with free fusion semirings in the sense of \cite{MR2511633} are simple. This generalizes \cite[Theorem 1]{chirvasitu:123509}, and provides new examples of simple quantum groups, arising naturally as cocenters. 

Finally, in \Cref{se.free_prod} we study free products of quantum groups in the sense of \cite{MR1316765} (i.e. quantum groups corresponding to coproducts of families of Hopf algebras; see below for details). The main results are \Cref{th.centers_free_prod}, which says that the center construction preserves free products, and \Cref{th.simplicity_free_prod}, stating that a free product of at least two non-trivial linearly reductive quantum groups has simple cocenter. This provides even more examples of simple quantum groups with highly non-commutative fusion semirings, as well as a systematic way of obtaining such examples.  

The proofs are elementary, in that they consist of little more than word combinatorics carried out in fusion semirings.

%%%%%%%%%%%%%%%%%%%%%%%%%%%%%%%%%%%%%%%%%%%%%%%%%%%%%%%%%%%%%%%%%%%%%%%%%%%%%%%%%%%%%%%%%%%%%%%%%%%%%%%%%%%%%%%%%%
\subsection*{Acknowledgement:} I would like to thank the anonymous referee for a very thorough reading of a previous draft, and the numerous suggestions on the improvement of both exposition and content.

%%%%%%%%%%%%%%%%%%%%%%%%%%%%%%%%%%%%%%%%%%%%%%%%%%%%%%%%%%%%%%%%%%%%%%%%%%%%%%%%%%%%%%%%%%%%%%%%%%%%%%%%%%%%%%%%%%
%%%%%%%%%%%%%%%%%%%%%%%%%%%%%%%%%%%%%%%%%%%%%%%%%%%%%%%%%%%%%%%%%%%%%%%%%%%%%%%%%%%%%%%%%%%%%%%%%%%%%%%%%%%%%%%%%%

\section{Preliminaries}\label{se.prel}

Although on occasion we specialize the discussion to compact quantum groups and hence complex Hopf $*$-algebras, the main objects throughout are cosemisimple Hopf algebras over some algebraically closed field $k$ (possibly of positive characteristic). Comodules are right and finite-dimensional unless specified otherwise, and the category of $C$-comodules is denoted by $\cM^C$. The notation pertaining to Hopf algebra or coalgebra theory is standard: $\Delta$, $\varepsilon$ and $S$ for comultiplications, counits and antipodes respectively (maybe with indices, as in $\Delta_C$ for the comultiplication of the coalgebra $C$), and Sweedler notation minus the summation sign for comultiplications and comodule structures, i.e. $\Delta(x)=x_1\otimes x_2$ and $v\mapsto v_0\otimes v_1$ respectively. For all of this and other background on Hopf algebras or coalgebras we refer to \cite{MR1786197,MR1243637}.

All coalgebras anywhere in this paper are, as mentioned before, \define{cosemisimple}, in the sense that their categories of comodules are semisimple. Since we are working over algebraically closed fields, this is the same as saying that the coalgebra is a direct sum of matrix coalgebras, i.e. duals of matrix algebras $M_n(k)$. To any $C$-comodule $V$ we associate the \define{coefficient subcoalgebra} $C_V\le C$ defined as the smallest subcoalgebra such that the comodule structure map $V\to V\otimes C$ factors through $V\otimes C_V$. This implements a bijection between (isomorphism classes of) simple $C$-comodules and matrix coalgebra summands of $C$. Moreover, if $C$ happens to be a Hopf algebra, the correspondence further agrees with the rest of the structure: $C_{V\otimes W}$ is exactly the product $C_VC_W$, while $C_{V^*}=S(C_V)$. We write $\widehat{C}$ for the set of isomorpism classes of simple comodules. 

One particularly nice way of being cosemisimple is to be a \define{CQG algebra}. We recall here only that these are complex Hopf $*$-algebras (i.e. $*$ is an antilinear, multiplication-reversing involution so that both $\Delta$ and $\varepsilon$ are $*$-algebra maps) with the aditional property that all $H$-comodules admit an inner product invariant under $H$ in some sense (which we will not need to make precise below). The idea is that such objects should behave like algebras of representative functions on a compact group, and the inner product condition mimics the existence of an invariant inner product for any finite-dimensional complex representation of a compact group. We refer to \cite{MR1310296}, where the notion first appeared, or alternatively to \cite[$\S$11]{MR1492989}, for all of the necessary background. As far as this paper is concerned, CQG algebras enter the picture by virtue of being cosemisimple Hopf algebras. 

By `quantum group' (sometimes, for clarity, `linearly reductive quantum group') we mean the object dual to a cosemisimple Hopf algebra. When the latter is a CQG algebra, we might call the dual object a `compact quantum group', as is customary in the literature. 

One notion that will play an important role is that of \define{fusion semiring} $R_+(H)$ for a Hopf algebra $H$. This is simply the Grothendieck semiring of its category of comodules, with addition induced by direct sums and multiplication induced by taking tensor products. As an additive monoid, it is free abelian on a set $\{r_V\}$ indexed by the simple $H$-comodules $V$. It has a natural multiplication-reversing involution $x\mapsto x^*$ induced by taking duals at the categorical level, i.e. $x^*$ is the representation contragredient to $x$ (the fact that $*$ is an involution follows from cosemisimplicity). We will often blur the line between a comodule and its class in the fusion semiring, and may write for instance $C_x$ for $C_V$ if $x\in R_+(H)$ is the class of the $H$-comodule $V$, may omit tensor products between comodules as in $VW$ instead of $V\otimes W$, etc. We will on occasion also refer to the fusion ring $R(H)$ of $H$: It is the Grothendieck ring (as opposed to the semiring) of the category of comodules. Both $R$ and $R_+$ are partially ordered by making $R_+$ the positive cone. In other words, $x\le y$ in $R(H)$ provided $y=x+z$ for some $z\in R_+(H)$. In $R_+$ this is the same as saying that the representation whose class $y\in R_+$ can be written as a direct sum of the representations corresponding to $x$ and $z$.  

There is a bijective correspondence between Hopf subalgebras of a Hopf algebra $H$ and \define{based} involution-invariant sub-semirings of $R_+(H)$. Here, `based' means being the span of a set of (classes of) simple comodules; since all sub-semirings in this paper are based in this sense, we will simply drop the adjective. Indeed, to a Hopf subalgebra $K\le H$ the (fully faithful) scalar corestriction functor on comodules induces an embedding $R_+(K)\to R_+(H)$, while in order to go in the other direction, to any (based) sub-semiring $R_+$ of $R_+(H)$ we can associate the direct sum of those matrix subcoalgebras of $H$ which correspond via $V\mapsto C_V$ to simples in $R_+$. 

As a final general remark complementing the previous paragraph, note that a Hopf subalgebra of a CQG algebra is automatically invariant under the $*$-structure, and hence the discussion above applies verbatim, without our having to worry about an additional $*$-invariance property.

%%%%%%%%%%%%%%%%%%%%%%%%%%%%%%%%%%%%%%%%%%%%%%%%%%%%%%%%%%%%%%%%%%%%%%%%%%%%%%%%%%%%%%%%%%%%%%%%%%%%%%%%%%%%%%%%%%

\subsection{Free fusion semirings}\label{subse.free}

These will be particularly amenable to the kind of analysis carried out in \Cref{se.simplicity}. Before explaining what freeness is, we take a short detour to fix some notation.

We will often consider structures $(R,*,\circ)$, where $R$ is a set, $*:R\to R$ is an involution, and $\circ:R\times R\to R\cup\{\emptyset\}$ is a map usually referred to as the fusion. Both $*$ and $\circ$ are extended to the free monoid $\langle R\rangle$ on $R$ (thought of as the set of words on the alphabet $R$ with juxtaposition as multiplication) by making $*$ multiplication-reversing, and by \[ (r_1\ldots r_k)\circ(s_1\ldots s_l) = r_1\ldots r_{k-1}(r_k\circ s_1)s_2\ldots s_l,\ r_i,s_j\in R \] respectively, where $r_1\ldots r_k$ and $s_1\ldots s_l$ are arbitrary words on the alphabet $R$, regarded as elements of the free monoid $\langle R\rangle$. The right hand side is $\emptyset$ whenever $r_k\circ s_1=\emptyset$.  

Following \cite[10.2]{MR2511633}, we have:

\begin{de}\label{def.free}
The fusion semiring of a cosemisimple Hopf algebra $H$ is said to be \define{free on $(R,*,\circ)$} if there is some indexing $\langle R\rangle\ni x\mapsto a_x\in \widehat{H}$ of the set of simple comodules by $\langle R\rangle$, with the monoidal unit corresponding to the empty word and the operation of taking duals corresponding to the involution $*$ on $\langle R\rangle$, and the multiplication in $R_+(H)$ is given by
\begin{equation}\label{eq.free} 
	a_xa_y = \sum_{x=vg,y=g^*w}(a_{vw}+a_{v\circ w}). 
\end{equation} The convention is that in the right hand side of this equality, the term $a_{v\circ w}$ is absent whenever $v\circ w=\emptyset$, and \define{free on $(R,*)$} means free on $(R,*,\circ)$ with $\circ$ constantly equal to $\emptyset$. 

`Free' period, unadorned by any data $(R,*,\circ)$, means that there is some such data which makes the semiring free in the sense of the previous paragraph. The same terminology applies to fusion rings. 
\end{de}

For Hopf algebras with free fusion semirings we usually identify simple comodules with words over the alphabet $R$, and denote the coalgebras $C_{a_v}$ by $C_v$ for words $v\in\langle R\rangle$.

\begin{re}\label{rem.prods_dominated_by_prods}
Keeping the notations from \Cref{def.free}, note that if $x=x't$ and $y=t^*y'$, then in the fusion semiring with its usual order we have $a_{x'}a_{y'}\le a_xa_y$. This follows from an application of \Cref{eq.free} to $a_{x'}a_{y'}$: 
\[
	a_{x'}a_{y'} = \sum_{x'=vh,y'=h^*w}(a_{vw}+a_{v\circ w}) = \sum_{x=vht,y=(ht)^*w}(a_{vw}+a_{v\circ w}),
\] 
the summations being over words $h\in\langle R\rangle$. In other words, in \Cref{eq.free} we are, roughly speaking, summing over all words $g$ ``common'' to $x$ and $y$ in some sense, whereas here we are summing only over those words among them which contain $t$ as as right hand segment. But then the resulting sum is clearly contained term by term in the right hand side of \Cref{eq.free}. 
\end{re}

\begin{re}\label{rem.free}
It will come in handy below to notice that a fusion ring free on $(R,*,\circ)$ in the sense of \Cref{def.free} is in fact free as a ring on the set $R$, i.e. it is a non-commutative polynomial ring in the variables $a_r$, $r\in R$. This can be shown by filtering the ring by lengths of words, and observing that \Cref{eq.free} says that $a_xa_y=a_{xy}$ up to shorter words. The argument is spelled out in \cite[3.2]{MR2917093}. 

One consequence of all of this is that a cosemisimple Hopf algebra with free fusion semiring cannot have non-trivial one-dimensional comodules, or equivalently, non-trivial grouplike elements; indeed, a one-dimensional comodule would correspond to an invertible element of the fusion ring, but the only invertibles in a non-commutative polynomial ring (over $\bZ$) are $\pm 1$. We will need this in \Cref{se.simplicity}.   
\end{re}

Several important families of cosemisimple Hopf algebras (in fact CQG algebras) exhibit the kind of behaviour axiomatized by \Cref{def.free}. We now recall some of these.

\begin{ex}\label{ex.free}
Let $Q$ be a positive self-adjoint $n\times n$ matrix, and $A_u(Q)$ the $*$-algebra freely generated by the $n^2$ elements $u_{ij}$, $i,j=\overline{1,n}$, subject to the relations demanding that both $u=(u_{ij})_{i,j}$ and $Q^{\frac 12}\overline uQ^{-\frac 12}$ be unitary as elements of $M_n(A_u(Q))$ (where $\overline u=(u_{ij}^*)_{i,j}$). 

$A_u(Q)$ can be made into a CQG algebra by declaring that $u_{ij}$ are the usual $n^2$ basis elements of a matrix coalgebra in the sense that \[ \Delta(u_{ij})=\sum_k u_{ik}\otimes u_{kj},\quad \varepsilon(u_{ij})=\delta_{ij} \] (a theme that will come up again and again in these examples; they are all defined by imposing relations on the $n^2$ matrix units of an $n\times n$ matrix coalgebra). 

The algebras $A_u(Q)$ were introduced by Wang and van Daele in \cite{MR1382726}, and they are the quantum analogues of unitary groups: Every finitely generated CQG algebra is a quotient of one of them, meaning, in dual language, that every ``compacy quantum Lie group'' embeds in the compact quantum group associated to $A_u(Q)$ for some $Q$.

When $n\ge 2$ (which will always be assumed), it was shown in \cite{MR1484551} that the fusion semiring of $A_u(Q)$ is free on $(\{\alpha,\alpha^*\},*)$, where $\alpha$ is the class of the fundamental representation corresponding to the matrix coalgebra spanned by $u$; it has a basis $e_i$, $i=\overline{1,n}$ and comodule structure defined by $e_j\mapsto \sum_i e_i\otimes u_{ij}$. Recall from \Cref{def.free} that this means $\circ$ is trivial, i.e. constantly $\emptyset$. 
\end{ex}

\begin{ex}\label{ex.ortho}
Now let $Q$ be a matrix with the property $Q\overline{Q}\in \bR I_n$ (just like the positivity assumption of the previous example, this avoids some redundance and degeneracy). The CQG algebra $B_u(Q)$ was also introduced in \cite{MR1382726} (denoted there by $A_o(Q)$) as the $*$-algbra obtained from $A_u(Q)$ by imposing the additional assumption that all $u_{ij}$ be self-adjoint. 

The fusion semiring of $B_u(Q)$ is determined in \cite{MR1378260}, and is isomorphic to that of $SU(2)$ whenever $n\ge 2$: free on $(\{\alpha\},*)$ for the fundamental representation $\alpha=\alpha^*$ defined as in the previous example. Again, $\circ$ is trivial.  
\end{ex}

\begin{ex}\label{ex.auto}
In \cite{MR1637425}, Wang introduced the quantum automorphism group $A^{aut}(B,\tau)$ of a traced finite-dimensional $C^*$-algebra $(B,\tau)$, meaning the initial object in the category of CQG algebras endowed with a coaction on $B$ that preserves the trace $\tau$ in the appropriate sense. If the trace $\tau$ is chosen judiciously, the fusion semiring of $A^{aut}$ is isomorphic to that of $SO(3)$ (as shown in \cite{MR1709109}). This means in particular that it is free on $(\{\alpha\},*,\circ)$, where $\alpha=\alpha^*$ and $\alpha\circ\alpha=\alpha$. 
\end{ex}

\begin{ex}\label{ex.reflection}
The quantum reflection groups of \cite{MR2511633} are another important family of examples. For a positive integer $n$ and an additional parameter $s$ which is either a positive integer or $\infty$, the quantum reflection group $A_h^s(n)$ is defined as the $*$-algebra freely generated by the $n^2$ generators $u_{ij}$, satisfying the conditions: 
\begin{enumerate}
\item $u$ and $\overline u$ are unitary.
\item All $u_{ij}$ are partial isometries, in the sense that $p_{ij}=u_{ij}u_{ij}^*$ are idempotent. 
\item $u_{ij}^s=p_{ij}$. 
\end{enumerate}
This is motivated by the fact that imposing the additional condition of commutativity would result in the algebra of functions on the group of $n\times n$ monomial matrices (i.e. one non-zero entry in each row and column) with $s$'th roots of unity as non-zero entries. 

In this setting (and assuming $n\ge 4$), \cite[7.3]{MR2511633} says that the fusion semiring of $A_h^s(n)$ is free on $(\bZ/s\bZ,*,\circ)$, where $x^*=-x$ and $x\circ y=x+y$ for $x,y\in\bZ/s\bZ$.  
\end{ex}

\begin{ex}\label{ex.free_product}
Given a family $H_i$ of cosemisimple Hopf algebras whose fusion semirings are free on $(R_i,*_i,\circ_i)$ for some index set $I\ni i$, their coproduct $H$ in the category of Hopf algebras is again cosemisimple, and its fusion semiring is free on $\left(R=\bigsqcup R_i,*,\circ\right)$ where this latter $*$ consists simply of putting together the involutions $*_i$ of the individual $R_i$'s, while for $r,s\in R$, $r\circ s$ is $r\circ_is$ if $r,s\in R_i$, and $\emptyset$ otherwise. 

Most of this follows for example from \cite[1.1]{MR1316765}, which basically says that the fusion semiring of the coproduct (there called the \define{free product} of the corresponding compact quantum groups) is the coproduct in the category of semirings of the $R_+(H_i)$. Although that paper is concerned with CQG algebras, these results extend easily to arbitrary cosemisimple Hopf algebras.
 
It will be worth our while to take somewhat of a detour, and go into the details of what happens to fusion rules when passing to coproducts. A note on terminology: In the sequel, `free product' refers to quantum groups (i.e. objects dual to Hopf algebras), whereas in the dual picture, when talking about algebras or Hopf algebras, we use `coproduct'. 
 
According to \cite{MR1316765}, the simples in $\cM^H$ are the tensor products $V_1\otimes\ldots\otimes V_n$ for all finite strings $(i_1,\ldots,i_n)$ of elements from $I$ such that $i_k\ne i_{k+1}$, $\forall k$ and all choices of non-trivial simples $V_k\in\cM^{H_{i_k}}$. The way to describe the fusion rules of $H$ briefly is to say that the canonical inclusions $\iota_i:R_+(H_i)\to R_+(H)$ make the right hand side the coproduct of the $R_+(H_i)$'s in the category of semirings. 

To be a bit more explicit, write the generic simple of $\widehat{H}\subset R_+(H)$ as $z=x_1\ldots x_n$ (strings meaning products, as announced above). Then, when computing its product with another such element, say $z'=y_1\ldots y_\ell$ ($y_k\in\widehat{H_{j_k}}$), three cases can occur: 

(i) If $i_n\ne j_1$, then $zz'$ is simply the stringing together of the $x$'s and $y$'s. 

(ii) If $i_n=j_1$ but $x_{i_n}$ is not the dual of $y_{j_1}$, then 
\begin{equation}\label{eq.fusion_free_prod2}
zz' = \sum_k n_k\ x_1\ldots x_{n-1}v_k y_2\ldots y_\ell,
\end{equation} 
where \[x_ny_1 = \sum_k n_k v_k,\quad n_k\in\bZ_+,\quad v_k\in\widehat{H_{i_n}}.\] 

(iii) Finally, if $x_n=y_1^*$ and \[x_ny_1 = 1 + \sum_k n_kv_k\] in $R_+(H_{i_n})$, then the product can be computed by recursion, with 
\begin{equation}\label{eq.fusion_free_prod3}
zz' = (x_1\ldots x_{n-1})(y_2\ldots y_\ell) + \sum_k n_k\ x_1\ldots x_{n-1}v_k y_2\ldots y_\ell
\end{equation} 
as the first step of that recursion.   

We leave the reader the task of applying this analysis to the case when $H_i$ have free fusion semirings on $(R_i,*_i,\circ_i)$, to obtain the claims made about the freeness of $R_+(H)$.
\end{ex}

%%%%%%%%%%%%%%%%%%%%%%%%%%%%%%%%%%%%%%%%%%%%%%%%%%%%%%%%%%%%%%%%%%%%%%%%%%%%%%%%%%%%%%%%%%%%%%%%%%%%%%%%%%%%%%%%%%

\subsection{Normal quantum subgroups}\label{subse.normal}

In the setting we are placing ourselves in, with Hopf algebras regarded as analogues of algebras of functions on groups, a normal quantum subgroup of a quantum group should correspond to a \define{quotient} Hopf algebra with some extra properties. We recall what these are, following \cite{MR1334152,MR2504527}.

\begin{de}\label{def.normal}
A morphism $\pi:H\to K$ of Hopf algebras is \define{normal} if the left and right kernels of $\pi$ defined respectively by
\[ \lker(\pi) = \{h\in H\ |\ \pi(h_1)\otimes h_2 = 1_K\otimes h\} \]
and
\[ \rker(\pi) = \{h\in H\ |\ h_1\otimes \pi(h_2) = h\otimes 1_K\} \]
coincide. 

For a quantum group (dual to the Hopf algebra) $H$, a \define{normal quantum subgroup} is a normal Hopf algebra surjection $\pi:H\to K$. The same discussion applies to CQG algebras, in which case all morphisms are always understood to be $*$-preserving. 
\end{de}

This is the definition adopted in \cite[2.1]{MR2504527} (see also \cite[1.1.5]{MR1334152}), and it is motivated by the fact that for Hopf algebras of representative functions on compact groups, it recovers the usual notion of normality (with the left and right kernels being algebras of functions on the right and left coset spaces respectively). Note that the map $\pi$ itself is regarded as the subgroup; two maps are always taken to represent the same subgroup whenever they are isomorphic as $H$-sourced arrows in the category of Hopf algebras (or CQG algebras). 

In the situation of \Cref{def.normal}, $P=\lker(\pi)=\rker(\pi)$ is a Hopf subalgebra of $H$, which can be thought of, morally, as functions on the quotient of the quantum group $H$ by the quantum subgroup $K$; we denote it by $\hker(\pi)$, and refer to it as the \define{Hopf kernel} of $\pi$. In all cases we deal with in this paper (where all Hopf algebras are cosemisimple), $\iota:P\to H$ and $\pi:H\to K$ are the two halves of what in \cite[1.2]{MR1334152} is called a \define{exact sequence} of quantum groups, and $\iota$ and $\pi$ in fact determine one another (see e.g. \cite[4.4]{MR2504527}). 

Consider the scalar corestriction functor $\cM^H\to\cM^K$ through $\pi$, turning an $H$-comodule $V$ with structure map $v\mapsto v_0\otimes v_1$ into a $K$-comodule via $v\mapsto v_0\otimes\pi(v_1)$. The Hopf kernel $\iota:P\to H$ of $\pi$ can now be characterized in the cosemisimple case as the sum of precisely those matrix subcoalgebras of $H$ whose corresponding simple comodules become trivial upon applying this functor (meaning they break up as direct sums of copies of the trivial one-dimensional comodule). 

As it turns out, there is a simple necessary and sufficient condition which will ensure that an inclusion $\iota:P\to H$ of (cosemisimple) Hopf algebras is half of an exact sequence $P\to H\to K$: It is that $P$ be invariant under either the left or the right adjoint action of $H$ on itself, defined by $h\triangleright h'=h_1h'S(h_2)$ and $h'\triangleleft h=S(h_1)h'h_2$ respectively (in which case we say that $P$ is \define{ad-invariant}). Indeed, one can show easily that left (right) kernels of Hopf algebra morphisms are always invariant under the right (resp. left) adjoint action; \cite[1.2.5]{MR1334152} ensures the opposite implication as soon as $H$ is faithfully flat over $P$, which according to \cite{2011arXiv1110.6701C} is always the case. 

Note also that starting with the ad-invariant Hopf subalgebra $P\le H$, the third term $K$ in the exact sequence $P\to H\to K$ (which is nothing but $H/HP^+$, where $P^+=\ker(P|_K)$) is automatically cosemisimple, also by \cite{2011arXiv1110.6701C}. In conclusion, knowledge of the ad-invariant Hopf subalgebra $\iota:P\to H$ is equivalent to giving a linearly reductive normal quantum subgroup $\pi:H\to K$, and we will often conflate these two points of view. Moreover, the correspondence $P\leftrightarrow K$ is size-reversing, in the sense that larger quotient Hopf algebras $K$ correspond to smaller Hopf subalgebras $P$.

\begin{re}\label{rem.nonzero}
We take a moment to observe that the (left, say) adjoint action of a matrix subcoalgebra $C$ of a Hopf algebra $H$ on another such subcoalgebra $D$ is non-zero. Indeed, This follows from the fact that $H\ni 1\in S(C)C$, and hence $D\le S(C)\triangleright(C\triangleright D)$. We need this below. 
\end{re}

Now that we know what normal quantum subgroups are, the following is natural:

\begin{de}\label{def.simple}
A quantum group with underlying Hopf algebra $H$ is \define{simple} if there are no normal quantum subgroups apart from the obvious ones, $\id:H\to H$ and $\varepsilon:H\to k$. 
\end{de}

\begin{re}\label{rem.def_comparison}
The notion from \cite{MR2504527} that best matches \Cref{def.simple} (in the case of compact quantum groups) is that of a so-called \define{absolutely} simple compact quantum group. The two definitions are still somewhat different though: Apart from not having any non-obvious normal quantum subgroups, there are additional requirements in \cite[3.3]{MR2504527} that the compact quantum group (a) be what is usually called a \define{matrix} CQG (in dual language, this amounts to the corresponding Hopf algebra being finitely generated as an algebra), (b) be \define{connected} (in the context of compact quantum groups, this means essentially that every non-trivial Hopf $*$-subalgebra is infinite-dimensional), and (c) have no non-trivial irreducible one-dimensional representations (i.e. comodules for the corresponding Hopf algebra). 

Conditions (b) and (c) hold automatically for the examples obtained below via \Cref{th.main}, all of whose fusion semirings are embedded into free ones. Indeed, the free fusion rule \Cref{eq.free} make it clear that any non-trivial Hopf subalgebra of such a Hopf algebra has infinitely many simple comodules, proving (b); on the other hand, \Cref{rem.free} ensures that (c) holds for Hopf algebras with free fusion semirings and their Hopf subalgebras.

As pointed out in \cite{chirvasitu:123509}, the examples covered by \Cref{th.main} will not, in general, satisfy condition (a). Other examples that do go the extra mile can be obtained by judiciously selecting finitely generated Hopf subalgebras (e.g. as in \cite[Proposition 10]{chirvasitu:123509}), but this is somehow less natural than the cocenter construction of the next section. All of this seemed to justify choosing the shorter definition over the longer, more refined one, as well as dropping the adjective `absolutely'. 
\end{re}

%%%%%%%%%%%%%%%%%%%%%%%%%%%%%%%%%%%%%%%%%%%%%%%%%%%%%%%%%%%%%%%%%%%%%%%%%%%%%%%%%%%%%%%%%%%%%%%%%%%%%%%%%%%%%%%%%%
%%%%%%%%%%%%%%%%%%%%%%%%%%%%%%%%%%%%%%%%%%%%%%%%%%%%%%%%%%%%%%%%%%%%%%%%%%%%%%%%%%%%%%%%%%%%%%%%%%%%%%%%%%%%%%%%%%

\section{Centers and cocenters}\label{se.center}

This section is devoted to defining the notion of center for a quantum group (regarded, as always in this paper, as the object dual to a cosemisimple Hopf algebra) and providing a characterization that will come in handy in \Cref{se.simplicity}. Roughly speaking, the center is, as expected, ``precisely the quantum subgroup which acts as scalars on any simple comodule''. Making this precise will require a bit of unpacking, but the idea is simple enough. The starting point is the result \cite[3.1]{MR2130607}, according to which for a compact group $G$, the character group of its center $Z(G)$ can be recovered from combinatorial data that can be read off from the fusion ring of $G$, or alternatively from its category of representations. Taking M\"uger's result as a guiding principle, we will \define{define} the center of a linearly reductive quantum group via the perfectly analogous process. The details follow.

We begin with the notion of centrality for a quantum subgroup $\pi:H\to K$ of a quantum group $H$, as introduced for instance in the course of the proof of \cite[4.5]{MR2504527}:

\begin{de}\label{def.central}
A morphism $\pi:H\to K$ of cosemisimple Hopf algebras is \define{central} if 
\begin{equation}\label{eq.central} 
	(\pi\otimes\id)\circ\Delta_H=(\pi\otimes\id)\circ\tau\circ\Delta_H: H\to K\otimes H, 
\end{equation} 
where $\tau$ is the usual flip of tensorands. 
\end{de}

\begin{re}
Note that unlike \cite{MR2504527} or \cite{2013arXiv1305.1080P} (where centrality and normal quantum subgroups also figure prominently), \Cref{def.central} does not demand that central morphisms be surjective. If $\pi:H\to K$ does happen to be both central and surjective, then $K$ is cocommutative, and hence a group algebra by cosemisimplicity. In addition, as noted in \cite[4.5]{MR2504527}, centrality implies normality. 
\end{re}

\begin{re}
Morphisms as in \Cref{def.central} are sometimes called `cocentral' (e.g. \cite[Definition 3.1]{2012arXiv1205.6110A}). The name `central' seems appropriate here, given the point of view we have adopted that Hopf algebras are to be regarded as function algebras.  
\end{re}

The following result ties in the definition with the representation-theoretic interpretation of centrality alluded to in the first paragraph of this section.

\begin{proposition}\label{pr.alt_char_centr}
A morphism $\pi:H\to K$ of cosemisimple Hopf algebras is central in the sense of \Cref{def.central} if and only if the associated scalar corestriction functor $\mathcal M^H\to\mathcal M^K$ turns every simple $H$-comodule into a direct sum of copies of a single one-dimensional $K$-comodule. 
\end{proposition} 
\begin{proof}
Since linear functionals $\varphi\in K^*$ separate the elements of $K$, \Cref{eq.central} can be stated as 
\begin{equation}\label{eq.central_bis}
	\varphi(\pi(h_1))h_2 = \varphi(\pi(h_2))h_1,\ \forall h\in H,\ \varphi\in K^*. 
\end{equation}
In turn, this is equivalent to the commutativity of all diagrams of the form 
\[
 \begin{tikzpicture}[anchor=base]
   \path (0,0) node (1) {$V$} +(3,0) node (2) {$V\otimes H$} +(0,-2) node (3) {$V$} +(3,-2) node (4) {$V\otimes H$}; 
   \draw[->] (1) -- (2);
   \draw[->] (3) -- (4);
   \draw[->] (1) -- (3) node[pos=.5,auto,swap] {$\scriptstyle \varphi$};
   \draw[->] (2) -- (4) node[pos=.5,auto] {$\scriptstyle \varphi\otimes\id$};
 \end{tikzpicture}
\]  
for all choices of a comodule $V\in\cM^H$ and a functional $\varphi\in K^*$, where the vertical maps represent the action of $\varphi$ on $V$ by $\varphi\triangleright v=\varphi(\pi(v_1))v_0$. Indeed, the right-down path in the diagram sends $v$ to $v_0\otimes \varphi(\pi(v_1))v_2$, whereas the down-right path sends it to $v_0\otimes \varphi(\pi(v_2))v_1$. 

In conclusion, the centrality of \Cref{def.central} is equivalent to $K^*$ acting by $H$-comodule maps on all $H$-comodules, or equivalently (by cosemisimplicity) on all simple $H$-comodules. But for simple $V\in\cM^H$, Schur's lemma makes this latter condition equivalent to $K^*$ acting on $V$ by $\varphi\triangleright v=\Psi(\varphi)(v)$ for some algebra homomorphism $\Psi:K^*\to k$. 

Since the $K^*$-module structure is induced by a $K$-comodule structure, in the situation described above $\Psi$ must be of the form $\varphi\mapsto \varphi(\gamma)$ for some $\gamma\in K$. The fact that $\Psi$ is an algebra map translates to $\gamma$ being a grouplike, and hence the $K$-comodule $V$ breaks up as a direct sum of $\dim V$ copies of the one-dimensional $k\gamma$-comodule.  
\end{proof}

We recall the following standard notion, for future reference:

\begin{de}\label{def.grading}
For a cosemisimple Hopf algebra $H$ and a group $\Gamma$, a \define{$\Gamma$-grading} of the fusion semiring $R_+=R_+(H)$ is a splitting $R_+=\bigoplus_{\gamma\in\Gamma}R_+^\gamma$ as additive abelian groups which respects the multiplicative structure,
\[
	R_+^\gamma R_+^{\gamma'}\subseteq R_+^{\gamma\gamma'},\ \forall \gamma, \gamma'\in\Gamma,
\]
and also the basis of irreducibles, in the sense that each $x\in\widehat{H}$ belongs to some $R_+^\gamma$. We will then say (as usual for rings graded by groups) that $x$ is homogeneous of degree $\gamma$ and write $\deg(x)=\gamma$. 

Since the entire structure consists of assigning elements of $\Gamma$ to simple $H$-comodules, we can also rephrase all of the above as a map $\deg:\widehat{H}\to\Gamma$ satisfying $\deg(x)=\deg(x_1)\ldots\deg(x_n)\in\Gamma$ whenever $x\le x_1\ldots x_n$ in $R_+$ (note that it suffices to impose this condition for $n=2$).  
\end{de}

\begin{re}\label{rem.tmap}
Note that a map as in the last paragraph of the definition is what M\"uger calls a \define{t-map} in \cite[2.1]{MR2130607} in the case of compact groups. 
\end{re}

Now let $\pi:H\to K$ be a central quantum subgroup. According to \Cref{pr.alt_char_centr}, every simple $V\in\cM^H$ can be written, when regarded as a $K$-comodule, as a direct sum of $\dim V$ copies of a $K$-comodule corresponding to some grouplike element $\gamma\in K$. Moreover, the surjectivity of $\pi$ means that $K$ is in fact a group algebra $k[\Gamma]$. Labelling $V$ by $\gamma$ is easily seen to define $\Gamma$-grading on the fusion semiring $R_+(H)$ in the sense of \Cref{def.grading}.

Conversely, suppose we have a grading of $R_+(H)$ by a group $\Gamma$. Then we can recover a central quantum subgroup $\pi:H\to k[\Gamma]$ as follows: For a simple $V\in\cM^H$ assigned degree $\gamma\in \Gamma$ by the grading, fix some basis $\{e_i\}$, $i=\overline{1,n}$, consider the corresponding matrix units $u_{ij}$ determined by $(e_j)_0\otimes (e_j)_1=\sum e_i\otimes u_{ij}$, and define the restriction of $\pi$ to $C_V$ by $u_{ij}\mapsto\delta_{ij}\gamma$. We leave to the reader the task of checking that this is a well-defined morphism of Hopf algebras (or CQG algebras, if $H$ is a CQG algebra and $\bC[\Gamma]$ is given its standard CQG structure making the elements of $\Gamma$ unitary). To check independence of the choice of basis $\{e_i\}$, for example, notice that a different choice would conjugate $u=(u_{ij})\in M_n(H)$ by a scalar $n\times n$ invertible matrix, and such a conjugation fixes the matrix $(\delta_{ij}\gamma)_{i,j}\in M_n(k[\Gamma])$ componentwise.

\begin{de}\label{def.gradings}
Let $H$ be a cosemisimple Hopf algebra. The \define{category of gradings} of $R_+(H)$ has as objects gradings of $R_+(H)$ by groups such that the simples are homogeneous, and as morphisms between a $\Gamma_1$ and a $\Gamma_2$-grading, morphisms of groups $f:\Gamma_1\to \Gamma_2$ respecting the grading, in the sense that if the simple $V$ has degree $\gamma_1\in \Gamma_1$ with respect to the first grading, then it has degree $f(\gamma_1)\in \Gamma_2$ with respect to the second one. 

The \define{category of central arrows} of $H$ has central arrows $H\to k[\Gamma]$ as objects, and the obvious commutative triangles as morphisms. 
\end{de}

\begin{re}\label{rem.grading}
Note that for any grading of $R_+(H)$ by a group $\Gamma$ as in the preceding discussion, the trivial comodule automatically gets degree $1$, and hence so do products $xx^*$ for $x\in\widehat{H}$. If $x$ has, say, degree $\gamma\in \Gamma$, $x^*$ will have degree $\gamma^{-1}$. In conclusion, the grading intertwines the operations $*$ and $\gamma\mapsto \gamma^{-1}$.    
\end{re}

With these preparations, the discussion preceding \Cref{def.gradings} can now be summarized as follows:

\begin{proposition}\label{pr.gradings}
For any cosemisimple Hopf algebra $H$, sending a central arrow $H\to k[\Gamma]$ to the grading described before \Cref{def.gradings} sets up an equivalence between central arrows of $H$ and gradings of $R_+(H)$.\qedhere  
\end{proposition}

Paraphrasing M\"uger's result \cite[3.1]{MR2130607}, mentioned briefly at the beginning of the current section, for a compact group $G$ with fusion semiring $R_+$, the category of gradings of $R_+$ has an initial object, and the corresponding group implementing this universal grading is precisely the character group of the center $Z(G)$. 

It is sensible now, in view of the result we've just cited and of \Cref{pr.gradings}, to try to define the center of the quantum group $H$ as the largest central quantum subgroup (i.e. initial object in the category of central arrows). \Cref{pr.gradings} allows us to do just that, by first observing that the category of gradings of $R_+(H)$ has an initial object. This universal grading has appeared, for example, in \cite{MR2097019,MR2130607,MR2383894}, sometimes in slightly different variants which can easily be adapted to the present situation; we recall several descriptions below.

\begin{de}\label{def.center_cocenter}
Let $H$ be a cosemisimple Hopf algebra. The \define{center} of the corresponding quantum group is the initial object $\pi:H\to Z$ of the category of central arrows. 

The \define{cocenter} is the Hopf kernel of the center $\pi$. 
\end{de}

Following \cite{MR2097019,MR2130607}, we denote the group implementing the initial grading of $R_+(H)$ by $\cC(H)$ (for `\define{chain} group', the term used in these papers for the object, in the context of compact groups). The two papers are concerned with (ordinary) compact groups, but the results of \cite{MR2130607} whose analogues we are interested in generalize quite easily, and we won't repeat the proofs. Instead, we content ourselves with giving the several descriptions of $\cC(H)$, and switch between them freely afterwards.   

A description of $\cC(H)$ analogous to that in \cite{MR2097019} would be as the set of equivalence classes of $\widehat H$, where $x,y\in \widehat H$ are declared equivalent (written $x\sim y$) whenever $x,y\le z_1\ldots z_n$ in $R_+(H)$ for some $z_j\in\widehat H$. Then all simple constituents of $xy$ are equivalent, so a multiplication operation descends to the set $\widehat{H}/_\sim$ of equivalence classes. It turns out this multiplication makes $\widehat{H}/_\sim$ into a group, with inverse given by taking duals, and the desired universal grading is $\widehat{H}\to\widehat{H}/_\sim$. 

Additionally, let $(\Gamma(H),\cdot)$ be the monoid generated by $\widehat{H}$, subject to relations $x\cdot y=z$ whenever $z\le xy\in R_+(H)$. $\Gamma(H)$ is in fact a group with inverse operation $x\mapsto x^*$, the map $\widehat{H}\to\Gamma(H)$ sending an irreducible to the corresponding generator of $\Gamma(H)$ is, once more, the universal grading. 

We refer the reader to the short paper \cite{MR2130607} for proofs of (most of) the claims. Although that paper does not explictly mention gradings, they are there in the guise of the t-maps mentioned in \Cref{rem.tmap}, via \Cref{def.gradings}. See also \cite{2013arXiv1305.1080P} for a somewhat different take on centers of compact quantum groups. Proposition 7.5 of that paper similarly identifies various possible descriptions of the center.

\begin{re}\label{rem.nat_induced_center_map}
We make here two simple observations, that will be used repreatedly (sometimes implicitly) below. 

First, for an inclusion $H'\le H$ of cosemisimple Hopf algebras, the composition $H'\le H\to K$ is central whenever the right hand arrow $H\to K$ is central. This follows immediately from \Cref{pr.alt_char_centr} and the fact that scalar corestriction through inclusions are fully faithful, and hence turn simples into simples. 

Secondly, inclusions $H'\le H$ also induce a group homomorphism $\cC(H')\to \cC(H)$. Indeed, by the previous observation, the composition $H'\to H\to k[\cC(H)]$ is central. But then, by the universality of $\cC(H')$, this composition must factor through $k[\cC(H')]\to k[\cC(H)]$ for a unique group morphism $\cC(H')\to \cC(H)$. 
\end{re}

\begin{re}
Because we are interested mainly in cosemisimple Hopf algebras, \Cref{def.center_cocenter} produces the largest central \define{linearly reductive} quantum subgroup. Any Hopf algebra also has a largest central quotient Hopf algebra, dually to \cite[Definition 2.2.3]{MR1382474}. Applying this construction to a cosemisimple Hopf algebra seems unlikely to result in a cosemisimple quotient however, hence the extra effort that goes into \Cref{def.center_cocenter}. 

For a CQG algebra $H$, however, the construction dual to \cite[2.2.3]{MR1382474} will indeed work: By the universality of the construction, the resulting quotient will be the largest among central $H$-quotients, but on the other hand it is automatically cosemisimple (and cocommutative, as all central quotients are), and hence a group algebra, because we are working over the algebraically closed field $\bC$. 
\end{re}

We now seek to describe the cocenter of a quantum group $H$ in terms of its fusion semiring. The preliminaries above almost take care of the first half of the following proposition, while the second half requires slightly more work and is needed in \Cref{se.simplicity}:

\begin{proposition}\label{pr.alt_descr}
Let $H$ be a cosemisimple Hopf algebra, and $\iota:P\to H$ its cocenter. The fusion semiring $R_+(P)$ can be described as either 

(a) The sub-semiring of $R_+(H)$ spanned by those $x\in\widehat H$ which are majorized by some product $(z_1\ldots z_n)(z_1\ldots z_n)^*\in R_+(H)$, $z_j\in\widehat H$ with respect to the usual partial order, or

(b) The sub-semiring of $R_+(H)$ generated by those $x\in\widehat H$ which are majorized by some $zz^*$, $z\in\widehat H$.
\end{proposition}
\begin{proof}
Recall first, from the discussion preceding \Cref{rem.nonzero}, that $\widehat{P}$ consists of precisely those ireducible comodules of $H$ which become trivial upon corestriction to the center of $H$, or, in other words, those irreducibles receiving degree $1$ in the universal grading. But since we've seen above that the universal grading is the quotient map $\widehat{H}\to\widehat{H}/\sim$, it follows that $R_+(P)\le R_+(H)$ is spanned by the simples $x\in\widehat H$ which are equivalent to $1$ through $\sim$. This means that $x$ and $1$ are majorized by some product $z_1\ldots z_n$ for $z_j\in\widehat H$, and hence $x\le (z_1\ldots z_n)(z_1\ldots z_n)^*$. Conversely, the latter condition on $x$ certainly implies $x\sim 1$, since both $x$ and $1$ are then summands of $(z_1\ldots z_n)(z_1\ldots z_n)^*$. 

The sub-semiring described in part (b) is a priori smaller than that of (a), and it is our task to show that the two in fact coincide. To this end, it is enough to show that any simple \begin{equation}\label{eq.alt_descr} x\le (z_1\ldots z_n)(z_1\ldots z_n)^* = z_1\ldots z_nz_n^*\ldots z_1^* \end{equation} is a summand of $(z_1\ldots z_{n-1})(z_1\ldots z_{n-1})^*yy^*$ for some $y\in\widehat H$, for we can then use induction on $n$. 

Consider $H$ as a comodule over itself, via the right adjoint coaction $h\mapsto h_2\otimes S(h_1)h_3$; more concretely, $H$ is simply the direct sum of all $y\otimes y^*$, $y\in\widehat H$. We leave it to the reader to check that for any $H$-comodule $V$, the back-and-forth maps between $H\otimes V$ and $V\otimes H$  defined by 
\[
	h\otimes v\mapsto v_0\otimes hv_1, \qquad v\otimes h\mapsto h\overline{S}(v_1)\otimes v_0
\]
are mutually inverse $H$-comodule isomorphisms, where $H$ has the adjoint coaction, $H\otimes V$ and $V\otimes H$ have the usual tensor product comodule structures, and $\overline{S}$ is the inverse of the antipode. In particular, we have \begin{equation}\label{eq.alt_descr_bis}
	z\otimes H\cong H\otimes z,\ \forall z\in R_+(H). 
\end{equation} 

To conclude, note that as a consequence of \Cref{eq.alt_descr}, $x$ is a summand of \[ (z_1\ldots z_{n-1})\otimes H \otimes(z_1\ldots z_{n-1})^*\ \cong\ (z_1\ldots z_{n-1})\otimes(z_1\ldots z_{n-1})^* \otimes H, \] where the isomorphism follows from \Cref{eq.alt_descr_bis} applied to $z=z_1\ldots z_{n-1}$. Finally, since $x$ is simple, it must be a subcomodule of some summand $(z_1\ldots z_{n-1})(z_1\ldots z_{n-1})^* yy^*$ of the latter.
\end{proof}

\begin{re}
More conceptually, the isomorphism $H\otimes V\cong V\otimes H$ exhibited in the proof says that $H$, with the right adjoint $H$-coaction and the right regular $H$-action, is a Yetter-Drinfeld module, and hence an element of the Drinfeld center of the category $\cM^H$. By definition, the Drinfeld center is the category consisting of $H$-comodules $X$ endowed with natural isomorphisms $X\otimes -\cong -\otimes X$ required to satisfy certain compatibility conditions. We refer the reader to \cite[10.6]{MR1243637} and \cite[XIII]{MR1321145} for these notions.  
\end{re}

Versions of the cocentral fusion sub-semiring $R_+(P)\le R_+(H)$, described as in part (b) of \Cref{pr.alt_descr}, have appeared in \cite{MR2183279,MR2383894,MR2130607}. The identification of the cocenter with this particular sub-semiring is a non-commutative analogue of the fact that in the classical case, for a compact group $G$, one recovers in this way precisely the fusion semiring of the quotient $G/Z(G)$ by the center (e.g. \cite[2.8]{MR2130607}).

\begin{re}
The pared-down discussion in \Cref{pr.alt_descr} is included here for the convenience of the reader, but let us mention that a more general version holds (and is essentially available in the literature), in the sense that sub-semirings defined as in (a) and (b) coincide for any so-called \define{fusion category} (even after dropping the usual requirement that there be only finitely many simple objects). 

Equation \Cref{eq.alt_descr_bis} still holds in this general setting if $H$ is taken, as in the text, to be the direct sum of all $y\otimes y^*$ for $y$ ranging over the simple objects. This is the gist of \cite[3.3, 3.4, 3.5]{MR2383894} (adapted to the case of infinitely many simples), which show that the initial grading of a fusion category coincides with the decomposition of the Grothendieck semiring of the entire category into summands as a module over the sub-semiring defined as in (b).    
\end{re}

%%%%%%%%%%%%%%%%%%%%%%%%%%%%%%%%%%%%%%%%%%%%%%%%%%%%%%%%%%%%%%%%%%%%%%%%%%%%%%%%%%%%%%%%%%%%%%%%%%%%%%%%%%%%%%%%%%
%%%%%%%%%%%%%%%%%%%%%%%%%%%%%%%%%%%%%%%%%%%%%%%%%%%%%%%%%%%%%%%%%%%%%%%%%%%%%%%%%%%%%%%%%%%%%%%%%%%%%%%%%%%%%%%%%%

\section{Simple cocenters}\label{se.simplicity}

As announced in the introduction, the main result of this section is:

\begin{theorem}\label{th.main}
For any cosemisimple Hopf algebra $H$ whose fusion semiring is free on a datum $(R,*,\circ)$, the cocenter $\iota:P\to H$ is simple in the sense of \Cref{def.simple}. 
\end{theorem}
\begin{proof}
As in \Cref{subse.free}, we denote the correspondence between words on $R$ and simple $H$-comodules by $\langle R\rangle \ni x\mapsto a_x\in\widehat H$. The statement reduces to showing that the only ad-invariant Hopf subalgebras $Q\le P$ are the two obvious ones. Throughout the rest of the proof, fix such a $Q$, and assume it is strictly larger than the scalars. Our aim will be to show that $Q=P$. 

According to part (b) of \Cref{pr.alt_descr}, it suffices to prove that $R_+(Q)\le R_+(P)$ contains $a_ua_{u^*}$ for any word $u$. Let $a_v\in R_+(Q)$ for some non-empty word $v$ (if there is no such $v$ then $Q$ is just the constants, and there is nothing to prove). By the fusion rules \Cref{eq.free}, $a_{vv^*}\in R_+(Q)$. Write $vv^*=r\ldots r^*$ for some $r\in R$, and consider the two possibilities:

(Case 1) The alphabet $R$ is of cardinality at least two. Then, I claim that for any word $u$, $R_+(Q)$ contains some simple $a_{u\ldots}$. Assuming this for now, note that $a_ua_u^*$ is a $Q$-comodule, being a summand of $a_{u\ldots}(a_{u\ldots})^*$: Apply \Cref{rem.prods_dominated_by_prods} to the case when $y^*=x$ is the larger word $u\ldots$, while $(y')^*=x'=u$. Proving this was our goal to begin with.

To prove the claim, let $s\ne r^*$ be some letter, $u$ an arbitrary word, and set $w=uu^*s^*s$. By the fusion rules \Cref{eq.free}, $a_w$ is a summand in $a_ua_u^*a_s^*a_s$, and hence belongs to the fusion subsemiring generated by $a_ua_u^*$ and $a_s^*a_s$. From \Cref{pr.alt_descr} (b) it follows that $a_w\in R_+(P)$. 

When expanding the product $a_wa_{vv^*}a_{w^*}$ using \Cref{eq.free}, the choice of $s$ as being different from $r^*$ imlies we get $a_{wvv^*w^*}$ and perhaps some lower terms involving the fusion of $w$ and $vv^*$ (or the fusion of $vv^*$ and $w^*$). In all of these terms, the fusion can at most eat up the last letter of $w$ and the first letter of $vv^*$ (or the last letter of $vv^*$ and the first letter of $w^*$, respectively). In any event, all resulting terms are of the form $a_{u\ldots}$, i.e. they are simples corresponding to words starting with $u$.

Now, since $Q\le P$ is assumed to be left ad-invariant, acting on $C_{vv^*}\le Q$ with $C_w\le P$ will produce a non-zero element of $Q$ by \Cref{rem.nonzero}. The previous paragraph implies that this element belongs to some simple coalgebra of the form $C_{u\ldots}$, which finishes the proof of the claim made at the beginning of Case 1.       

(Case 2) $R$ is a singleton $\{r\}$. In this case the duality involution $*$ is necessarily the identity, and there are only two possibilities for the fusion $\circ$: $r\circ r$ is either $\emptyset$ or $r$. These are precisely the fusion semirings from \Cref{ex.ortho,ex.auto} respectively, and the desired result is essentially proven in \cite[4.1, 4.7]{MR2504527} (Wang's statements refer to the quantum groups themselves, but the proofs rely only on the structure of their fusion semirings). 
\end{proof}

\begin{re}\label{rem.def_comparison_bis}
In the above proof, part (1) can be tweaked to get a somewhat stronger result: With the notations of the theorem, the quantum group associated to \define{any} Hopf subalgebra $P'\le P$ is simple, provided among the simple comodules of $P'$ we can find two, say $a_x$, $a_y$, for words $x,y\in \langle R\rangle$ ending in different letters. In particular, there is a rich supply of finitely generated Hopf subalgebras with simple underlying quantum groups (see \Cref{rem.def_comparison} above).  
\end{re}

Call a quantum group with simple cocenter \define{projectively simple}. In view of \Cref{ex.free,ex.reflection}, free unitary groups and quantum reflection groups are all projectively simple. As noted before, this generalizes the main result of \cite{chirvasitu:123509}. The following simple observation is of some use in determining what the cocenters of the theorem look like.

\begin{proposition}\label{pr.center_of_free}
Let $H$ be a cosemisimple Hopf algebra with fusion semiring free on the data $(R,*,\circ)$. Then, the universal grading group $\mathcal{C}(H)$ is generated by $R$, subject to the relations turning $*$ into the inverse and $\circ$ into the multiplication. 
\end{proposition}
\begin{proof}
According to the discussion immediately preceding \Cref{rem.nat_induced_center_map}, the universal grading group can be described as being generated by $\widehat{H}$, with relations $x\cdot y=z$ whenever $z\le xy\in R_+(H)$ for irreducibles $x,y,z$. But by our free fusion \Cref{eq.free}, $z\le xy$ implies $z=x\circ y$ when $y\ne x^*$, and $z=1$ or $z=x\circ x^*$ when $y=x^*$. The resulting relations are the ones in the statement.  
\end{proof}

Let's see what this says about some of the examples from \Cref{subse.free}.

\begin{ex}\label{ex.free_center}
Recall from \Cref{ex.free} that the fusion semirings of the free unitary groups $A_u(Q)$ (for some positive $n\times n$ matrix $Q$, $n\ge 2$) are free on $R=\{\alpha,\alpha^*\}$. It follows from \Cref{pr.center_of_free} that the grading group $\mathcal{C}(A_u(Q))$ is $\bZ$, and hence the center is precisely the central subgroup $A_u(Q)\to \bC[t,t^{-1}]$, $u_{ij}\mapsto\delta_{ij}t$ of \cite[4.5]{MR2504527}. It is now a simple exercise to show that the simple comodules of the cocenter are those indexed by words in $\alpha,\alpha^*$ having equal numbers of $\alpha$'s and $\alpha^*$'s (see e.g. \cite[Lemma 7]{chirvasitu:123509}).  
\end{ex}

\begin{ex}\label{ex.reflection_center}
According to \Cref{ex.reflection}, the fusion semirings $R_+(A_h^s(n))$ (for $n\ge 4$, which we henceforth assume) are free on data $(\bZ/s\bZ,*,\circ)$ which precisely imitates the additive group structure of $\bZ/s\bZ$: $x^*=-x$ and $x\circ y=x+y$. \Cref{pr.center_of_free} then implies that the center is $A_h^s(n)\to\bC[t,t^{-1}]/(t^s=1)$ (when $s=\infty$ there is no relation at all), defined as before by $u_{ij}\mapsto \delta_{ij}t$. The simple comodules of the cocenter are those indexed by words $r_1\ldots r_k\in\langle \bZ/s\bZ\rangle$ satisfying $r_1+\ldots+r_k=0\in\bZ/s\bZ$. 

As in the previous example, the cocenter is large enough to allow, via \Cref{rem.def_comparison_bis}, for many examples of finitely generated Hopf algebras with simple underlying compact quantum groups (at least when $s\ge 2$).  

Moreover, note that since the normal quantum subgroups of $A_h^s(n)$ are all finite, it has no non-trivial normal \define{connected} quantum subgroups (see \Cref{rem.def_comparison}). It also satisfies Wang's conditions (a), (b) and (c) listed in the same remark above, and hence meets his definition of a simple compact quantum group (as opposed to his stronger condition of absolute simplicity, recalled in \Cref{rem.def_comparison}). Indeed, $A_h^s(n)$ is finitely generated as an algebra (condition (a)) by definition, while \Cref{rem.def_comparison} argues that conditions (b) and (c) are always satisfied by quantum groups with free fusion semirings. 
\end{ex}

\begin{ex}\label{ex.free_product_center}
When applied to \Cref{ex.free_product}, \Cref{pr.center_of_free} shows that the universal grading group construction $\cC(\bullet)$ preserves coproducts when applied to cosemisimple Hopf algebras \define{with free fusion semirings}. The phrase in italics can in fact be dropped, as we will see in the next section. 
\end{ex}

%%%%%%%%%%%%%%%%%%%%%%%%%%%%%%%%%%%%%%%%%%%%%%%%%%%%%%%%%%%%%%%%%%%%%%%%%%%%%%%%%%%%%%%%%%%%%%%%%%%%%%%%%%%%%%%%%%
%%%%%%%%%%%%%%%%%%%%%%%%%%%%%%%%%%%%%%%%%%%%%%%%%%%%%%%%%%%%%%%%%%%%%%%%%%%%%%%%%%%%%%%%%%%%%%%%%%%%%%%%%%%%%%%%%%

\section{Projective simplicity of free products}\label{se.free_prod}

Here, we take a closer look at coproducts of cosemisimple Hopf algebras, studied in \cite{MR1316765} and recalled above in \Cref{ex.free_product}. After giving a description of the center of such a coproduct in terms of the centers of the individual Hopf algebras (\Cref{th.centers_free_prod}), we show in \Cref{th.simplicity_free_prod} that coproducts are automatically projectively simple (given some mild non-triviality conditions). 

To fix notation, we will be working with families $H_i$, $i\in I$ of cosemisimple Hopf algebras and their coproduct $H$ in the category of Hopf algebras (or equivalently algebras, or $*$-algebras, or CQG algebras, etc.; all of these notions coincide when they make sense for the objects in question). 

The discussion in \Cref{ex.free_product} will be implicit throughout the entire section, and we will use the same notation and setup repeatedly. We also reprise the notation $x\mapsto a_x$ for the correspondence between words and simples, when we want to distinguish between concatenation of words and products in fusion semirings.

\begin{theorem}\label{th.centers_free_prod}
Let $H_i$, $i\in I$ be a set of cosemisimple Hopf algebras, and $\displaystyle H=\coprod_i H_i$ their coproduct. The canonical morphisms $\cC(H_i)\to\cC(H)$ obtained by applying \Cref{rem.nat_induced_center_map} to the inclusions $H_i\le H$ make the universal grading group of $H$ the coproduct of the groups $\cC(H_i)$.
\end{theorem}
\begin{proof}
For an arbitrary Hopf algebra $H$ and a group $\Gamma$, denote by $\cat{Cent}(H,\Gamma)$ the set of central arrows $H\to k[\Gamma]$. By \Cref{def.center_cocenter}, the universal grading group $\cC(H)$ is the object representing the functor $\cat{Cent}(H,-):\cat{Gp}\to\cat{Set}$ from the category of groups to that of sets, sending a group $\Gamma$ to $\cat{Cent}(H,\Gamma)$:
\begin{equation}\label{eq.added} 
	\cat{Cent}(H,\Gamma)\cong\cat{Hom}_{\cat{Gp}}(\cC(H),\Gamma) \text{ functorially in } \Gamma. 
\end{equation} 

Going back to the setting of the theorem, $\cC(H)$ represents the functor 
\begin{equation}\label{eq.centers_free_prod}
	\cat{Cent}(H,-) \cong \cat{Cent}\left(\coprod_i H_i,-\right)\cong \prod_i\cat{Cent}(H_i,-)\cong \prod_i\cat{Hom}_{\cat{Gp}}(\cC(H_i),-),
\end{equation}
where the products and isomorphisms are in the category of functors $\cat{Gp}\to\cat{Set}$, and: 

The first natural isomorphism is just the definition of $H$.

To verify the second isomorphism, note that we already know there is an identification, natural in $\Gamma$, between the set of Hopf algebra maps $H\to k[\Gamma]$ and the set of tuples of Hopf algebra maps $H_i\to k[\Gamma]$, because $H$ is the coproduct of the $H_i$'s in the category of Hopf algebras. The desired isomorphism then simply states that a Hopf algebra map $H\to k[\Gamma]$ is in addition central if and only if the corresponding compositions $H_i\to H\to k[\Gamma]$ are all central. This follows from the fact that centrality for a map out of $H$ can be checked on a set of algebra generators, and the $H_i$'s generate $H$ as an algebra.

The third isomorphism is the universal property of the individual grading groups $\cC(H_i)$.

The last functor in the isomorphism chain \Cref{eq.centers_free_prod} is represented by the coproduct of the groups $\cC(H_i)$ (by the very definition of a coproduct; e.g. \cite[III.3]{MR1712872}). On the other hand, the left hand side of \Cref{eq.centers_free_prod} is represented by $\cC(H)$, according to \Cref{eq.added}. By the uniqueness-up-to-isomorphism of a representing object (a consequenc eof Yoneda's lemma; \cite[III.2]{MR1712872}), the coproduct of the groups $\cC(H_i)$ is isomorphic to $\cC(H)$. Checking that the coproduct structure map $\cC(H_i)\to \cC(H)$ we've exhibited via \Cref{eq.centers_free_prod} is precisely the canonical one from \Cref{rem.nat_induced_center_map} is now more or less tautological. 
\end{proof}

\begin{re}
Although the proof of the theorem uses Hopf algebra maps and such, thus looking like it is very much specific to Hopf algebras, one could rephrase it so as to get a stronger result: The universal grading group of a coproduct of fusion rings in the sense of \cite{MR2383894} (where the term is used differently) is the coproduct of the individual universal grading groups.   
\end{re}

\begin{re}
The result can be regarded as a ``free'' analogue of the fact that taking the center of product of groups is the product of the individual centers. 
\end{re}

The next result says that free products provide a rich supply of simple quantum groups by the same natural cocenter construction.

\begin{theorem}\label{th.simplicity_free_prod}
Let $H_i$, $i\in I$ be a family of at least two non-trivial cosemisimple Hopf algebras, and $\displaystyle H=\coprod H_i$ their coproduct. The quantum group correpsonding to $H$ is projectively simple.    
\end{theorem}

This provides a rather strong answer to Problems 4.3 and 4.4 of \cite{Wang_new}, concerned with the simplicity of coproducts when the individual Hopf algebras are simple; the theorem dispenses with this last requirement entirely, with only the minor caveat that one must mod out the center. 

Before going into the proof, let us record some consequences. The following one is immediate by \Cref{th.centers_free_prod,th.simplicity_free_prod}:

\begin{corollary}\label{cor.simplicity_free_prod}
The free product of at least two non-trivial, centerless, linearly reductive quantum groups is simple.\qedhere
\end{corollary}

\begin{re}
The corollary refers to the notion of simplicity from \Cref{def.simple}. A moment's thought, however, should convince the reader that the result remains true if one adopts the definition \cite[3.3]{MR2504527}, provided we restrict ourselves to free products of finite families, if we insist on the finite generation of our Hopf algebras (see \Cref{rem.def_comparison} above or \cite[Remark 5]{chirvasitu:123509} for a comparison). 
\end{re}

\begin{corollary}
For $H_i$ and $H$ as in the statement of \Cref{th.simplicity_free_prod}, any proper normal quantum subgroup of $H$ is central. 
\end{corollary}
\begin{proof}
According to the discussion preceding \Cref{rem.nonzero}, taking the associated quotient quantum group of a normal quantum subgroup implements a size-reversing one-to-one correspondence between normal, cosemisimple quotient Hopf algebras and ad-invariant Hopf subalgebras of $H$. Hence, the statement of the corollary can be rephrased as saying that any non-trivial ad-invariant Hopf subalgebra $\iota:Q\to H$ contains the cocenter $P\le H$. 

If $Q$ intersects the cocenter non-trivially, the intersection (being ad-invariant in $P$) must equal $P$ by \Cref{th.simplicity_free_prod}. Let us now assume that at least one of the $H_j$'s, say $H_{i_1}$, is not a group algebra (or there is nothing to prove). I claim that in this case, $P\cap Q$ is non-trivial, and hence we are in the situation we've already covered. 

To prove the claim, first suppose some $H_i$, $I\ni i\ne i_1$ has a simple comodule $y$ of dimension greater than $1$, and let $z=x_1\ldots x_n\in\widehat{Q}$ with $x_k\in H_{i_k}$. Acting on $C_z$ by $C_y$ via the left adjoint action will produce elements of matrix coalgebras associated with simple comodules of the form $y\ldots$ (i.e. words on $\bigcup\widehat{H_j}$ starting with $y$), some of which are non-zero by \Cref{rem.nonzero}. It follows that $Q$ has some simple $z'=y\ldots$ of dimension greater than $1$ (because $\dim y>1$), and $z'(z')^*\in R_+(Q)\cap R_+(P)$ will then contain some non-trivial summand. 

The only case left to treat is that when all $H_i$ are group algebras except for $H_{i_1}$. First act on $C_z$ by some non-trivial grouplike $x$ in an $H_i$, $i\ne i_1$, and then apply the previous discussion to the resulting simple $x\ldots$, whose underlying matrix coalgebra can now be acted upon by $C_y$ for some simple $y\in\widehat{H_{i_1}}$ of dimension at least $2$. 
\end{proof}

\begin{proof of simplfree}
Let $P\le H$ be the cocenter of $H$, and $Q\le P$ a non-trivial ad-invariant Hopf subalgebra. Our task is to show that in fact $Q=P$. It is implicit here that $P$ is non-trivial, i.e. $H_i$ are not all group algebras (otherwise there would be nothing to prove). We use the same notation as before for simple comodules of $H$, regarding them as words on the alphabet $\bigcup\widehat{H_i}$. So an expression such as $\ldots x$, $x\in\widehat{H_i}$ means word ending with $1\ne x\in\widehat{H}$. We might also refer to such a word simply as ending in $\widehat{H_i}$ (and similarly for `starting' instead of `ending'). 

{\bf (Step 1)} For any simple $u=u_1\ldots u_n\in\widehat{H}$ with $u_k\in\widehat{H_{i_k}}$, $P$ has a simple $u\ldots u^*$ having $u$ as an initial segment and $u^*$ as a final segment. Indeed, if at least one $H_j$, $j\ne i_n$ is not a group algebra, the word $uyu^*$ for some non-trivial simple $y$ of the cocenter of $H_j$ will do. Otherwise (in other words, if every $H_j$, $j\in I$ is a group algebra except for $H_{i_n}$), let $y$ be a non-trivial simple of the cocenter of $H_{i_n}$ and $x$ a non-trivial simple of some $H_j$, $j\ne i_n$, and consider the word $uxyx^*u^*$ instead. 

{\bf (Step 2)} There is an $i$ such that $Q$ has a non-trivial simple simple starting and ending in $\widehat{H_i}$. To see this, let $z = x_1\ldots x_n\in\widehat{Q}$ be an arbitrary word, with $1\ne x_k\in\widehat{H_{i_k}}$, and assume without loss of generality that $i_1\ne i_n$ (or we would be done). If $|I|\ge 3$, acting on $C_z$ by $C_v$ for some $v\in\widehat{P}$ starting in $H_i$, $i\ne i_1,i_n$ will do the trick, in the now-familiar fashion (such a $v$ exists by Step 1). So we are reduced to $I=\{1,2\}$.  

Furthermore, by the fusion rules \Cref{eq.fusion_free_prod2,eq.fusion_free_prod3}, the only way in which $a_za_z^*\in R_+(Q)$ could fail to have some non-trivial summand would be for the letters of $z$ to all correspond to one-dimensional comodules of the various $H_i$'s. Assume all $x_i$'s are one-dimensional (or we would be done).  

We may now further assume without loss of generality that $H_{i_n}$ is not a group algebra (since $\{i_1,i_n\} = \{1,2\}$), and in particular, that there is some non-trivial $y\in\widehat{H_{i_n}}$, $y\ne x_{i_n}^*$. Now act on $C_z$ with $C_v$ for some $v=y\ldots y^*\in\widehat{P}$ (Step 1 again) to obtain some word in $\widehat{Q}$ starting in $y$ and ending in $y^*$.  

{\bf (Step 3)} Same as Step 1, with $Q$ instead of $P$. Let $1\ne z\in \widehat{Q}$ be as in Step 2, starting and ending in the same $\widehat{H_i}$, and $v=u\ldots u^*\in\widehat{P}$ a simple as in Step 1, with $u=u_1\ldots u_n$, $u_k\in\widehat{H_{i_k}}$. If $i\ne i_1$, then acting on $C_z\le Q$ via the left adjoint action with elements of $C_v$ we get elements (some non-zero, according to \Cref{rem.nonzero}) of the coalgebra $C_{vzv^*}$ associated to a simple of the desired form. Otherwise, i.e. if $z$ does start in $\widehat{H_{i_1}}$, then first act on $C_z$ with $C_{v'}$ for some $v'=u'\ldots(u')^*\in\widehat{P}$ with $u'$ not starting in $H_{i_1}$, and then act further on $C_{v'z(v')^*}$ with $C_v$. In both cases, ad-invariance ensures that $C_{vzv^*}$ or, respectively, $C_{vv'z(vv')^*}$ is contained in $Q$.   

{\bf (Step 4: End of the proof)} Let $u$ be an arbitrary word on $\bigcup\widehat{H_i}$. According to the previous step, we can find a simple comodule $z=u\ldots\in\widehat{Q}$. The product $a_za_z^*$ in $R_+(Q)$ majorizes $a_ua_u^*$, and hence $R_+(Q)$ contains $a_ua_u^*$ for any $u\in\widehat{H}$. The conclusion follows from \Cref{pr.alt_descr} (b). 
\end{proof of simplfree}

Finally, let us end by noting that most examples of simple quantum groups provided by \Cref{th.simplicity_free_prod} and in the previous section are far from being classical in yet another way: They lack what Wang calls property F, which demands that Hopf subalgebras be ad-invariant automatically (this is equivalent to \cite[3.6]{MR2504527}). For a free product $H$ of at least two non-trivial quantum groups $H_i$, $i\in I$, for instance, the cocenter of some non-group algebra $H_i$ is contained in the cocenter of $H$, but is not ad-invariant by \Cref{th.simplicity_free_prod} (or just by the definition of the coproduct $H=\coprod H_i$). This addresses the remark of \cite[$\S$5]{MR2504527} to the effect that all examples of simple compact quantum groups presented there have property F.

%%%%%%%%%%%%%%%%%%%%%%%%%%%%%%%%%%%%%%%%%%%%%%%%%%%%%%%%%%%%%%%%%%%%%%%%%%%%%%%%%%%%%%%%%%%%%%%%%%%%%%%%%%%%%%%%%%
%%%%%%%%%%%%%%%%%%%%%%%%%%%%%%%%%%%%%%%%%%%%%%%%%%%%%%%%%%%%%%%%%%%%%%%%%%%%%%%%%%%%%%%%%%%%%%%%%%%%%%%%%%%%%%%%%%

%\bibliography{cocenters}{}
%\bibliographystyle{plain}

\end{document}